\documentclass{amsart} 
\usepackage{amssymb,latexsym,textcomp}
\usepackage{xcolor}

\usepackage[OT2,T1]{fontenc}

\UseRawInputEncoding

\theoremstyle{plain}
\newtheorem{theorem}{Theorem}[section]
\newtheorem{lemma}[theorem]{Lemma}
\newtheorem{corollary}[theorem]{Corollary}
\newtheorem{proposition}[theorem]{Proposition}

\theoremstyle{definition}
\newtheorem{definition}{Definition}[section]
\newtheorem{example}{Example}[section]

\theoremstyle{remark}
\newtheorem{remark}{Remark}[section]

\DeclareMathOperator{\diam}{diam}

\newcommand{\skp}[2]{\left<#1,#2\right>}

\newcommand{\strongperp}{\mathrel{\perp\!\!\!\!\perp}^S}
\newcommand{\strongp}{\mathrel{\perp}^S}

\numberwithin{equation}{section} 

\begin{document}
\title[Isolated vertices and diameter of the orthograph]{Isolated vertices and diameter of the $BJ$-orthograph in $C^*$-algebras} 

\author{Dragoljub J. Ke\v cki\' c}
\address{University of Belgrade\\ Faculty of Mathematics\\ Student\/ski trg 16-18\\ 11000 Beograd\\ Serbia}

\email{keckic@matf.bg.ac.rs}

\author{Srdjan Stefanovi\' c}
\address{University of Belgrade\\ Faculty of Mathematics\\ Student\/ski trg 16-18\\ 11000 Beograd\\ Serbia}

\email{srdjan.stefanovic@matf.bg.ac.rs}

\thanks{The research was supported by the Serbian Ministry of Education, Science and Technological Development through Faculty of Mathematics, University of Belgrade.}

\begin{abstract} We give necessary and sufficient condition that an element of an arbitrary $C^{*}$-algebra is an isolated vertex of the orthograph related to the mutual strong Birkhoff-James orthogonality. Also, we prove that for all $C^{*}$-algebras except $\mathbb{C},\mathbb{C}\oplus\mathbb{C}$ and $M_2(\mathbb{C})$ all non isolated points make a single connected component of the orthograph which diameter is less than or equal to $4$, i.e.\ any two non isolated points can be connected by a path with at most $4$ edges. Some related results are given.
\end{abstract}




\maketitle

\section{Introduction}

Birkhoff-James orthogonality in Banach spaces has been studied for many years. It is defined as follows.

\begin{definition} Let $X$ be a Banach space, and let $x$, $y\in X$. We say that $x$ is Birkhoff-James orthogonal to $y$, and denote $x\perp_{BJ} y$, or simply $x\perp y$, if for any $\lambda\in \mathbb C$ there holds
\begin{equation}\label{BJord}
\|x+\lambda y\|\ge\|x\|.
\end{equation}

\end{definition}

It is easy to see that (\ref{BJord}) implies orthogonality in the usual sense, if $X$ is equipped with an inner product generating the norm. If, however, the norm on $X$ does not satisfy the parallelogram rule, i.e.\ $X$ is not a Hilbert space, Birkhoff-James orthogonality has some deficiencies. For instance, it is neither symmetric nor additive. Birkhoff-James orthogonality we shall abbreviate to $BJ$-orthogonality in further.

For further details on $BJ$-orthogonality, the reader is referred to a recent survey \cite{ArambasicSurvey} and references therein.

In this note we are interested in $C^*$-algebras. Besides usual $BJ$-orthogonality, there is also a notion of strong $BJ$-orthogonality.

\begin{definition} Let $A$ be a $C^*$-algebra and let $a$, $b\in A$.

a) We say that $a$ is strong $BJ$-orthogonal to $b$, and denote $a\strongp b$ if for any $c\in A$ there holds
\begin{equation}\label{BJstrong}
\|a+bc\|\ge\|a\|.
\end{equation}

b) We say that $a$ and $b$ are mutually strong $BJ$-orthogonal to each other, and denote $a\strongperp b$ if $a\strongp b$ and $b\strongp a$.
\end{definition}

\begin{remark}
    The part a) of the previous definition is from \cite{ArambasicAFA2014}, and part b) from \cite{ArambasicBJMA2020}.
\end{remark}

\begin{remark} Obviously, the previously defined notion does not depend on scalar multiple.
\end{remark}

\begin{remark}
    The notion of strong $BJ$-orthogonality actually comes from the investigation of Hilbert modules, i.e.\ modules over some $C^*$-algebra. It is usual to consider left modules, where scalars are written to the right of a vector. Any $C^*$ algebra can be regarded as a Hilbert module overitself. Therefore, (\ref{BJstrong}) is a special case of $BJ$-orthogonality on Hilbert modules.
\end{remark}

Recently, in \cite{ArambasicBJMA2020} the orthograph related to $BJ$-orthogonality is introduced.

\begin{definition} Let $A$ be a $C^*$-algebra, and let $S$ be the corresponding projective space, i.e.\ $S=(A\setminus\{0\})/a\sim \lambda a$, where $\lambda$ runs through $\mathbb C$.

The orthograph $\Gamma(A)$ related to the mutual strong $BJ$-orthogonality is the graph with $S$ as a set of vertices and with edges consisting of those pairs $(a,b)\in S\times S$ for which $a\strongperp b$.
\end{definition}

Among others, the following results are proved in \cite{ArambasicBJMA2020}:

R1. In unital $C^*$ algebra, any right invertible element is an isolated vertex of the orthograph.

R2. All non right invertible elements form a connected component of the orthograph, in the following $C^*$ algebras: $M_n(\mathbb C)$, $n\ge3$, $B(H)$ for infinite dimensional Hilbert space $H$ and $C(K)$. On $C^*$ algebra $M_2(\mathbb C)$, any non isolated vertex belongs to some two-membered connected component.

R3. The diameter of the orthograph is also investigated. Given two vertices in some connected component, their distance is the minimal number of edges that connect them. The diameter of the orthograph (or its connected component) is the maximal distance of arbitrary two connected vertices. The following answers were given. For $M_3(\mathbb C)$ the diameter is $4$. For $M_n(\mathbb C)$, $n\ge 4$ and $B(H)$ with $H$ infinite dimensional, the diameter is $3$. For $C(K)$ where $K$ does not contain point with countable local bases, the diameter is $2$. Otherwise, the diameter of $C(K)$ is $3$.

In this note we consider the following questions related to previously listed results.

Q1. Whether the converse of R1 is true: "any isolated vertex in a unital $C^*$-algebra is right invertible"? What can be said about nonunital $C^*$ algebras? (This question is posed in \cite[Section 6]{ArambasicBJMA2020}.)

Q2. Can we estimate the diameter of $\Gamma(A)$ in general case?

Q3. Is $M_2(\mathbb C)$ the only $C^*$ algebra where non invertible elements do not make a connected component?

In section \ref{Isolated} we give complete answer to Question 1. In section \ref{NonComTop} we give a brief survey of Akemann's noncommutative topology, which is the main tool in the next section. In section \ref{Diameter} we give affirmative answers to Question 3 and give a sharp estimate of the diameter in general case (Question 2). In section \ref{finite case} we discuss the diameter of finite dimensional $C^*$ algebras. Finally, in section \ref{questions} we raise two questions for future investigation.

It was supposed that the reader is familiar with basic $C^*$ algebra notions and techniques, e.g.\ states, pure states, Kadison transitivity theorem, etc. Advanced techniques, those related to the second dual of a given $C^*$ algebra is presented in more details. For both basic and advance facts on $C^*$-algebras, the reader is reffered to books \cite{Murphy} and \cite{Pedersen}.

\subsection{Reducing to positive elements}

Before we state and prove our results, in next two Lemmata we prove that $BJ$-orthogonality between $a$ and $b$ depends only on the minimal subalgebra containing $a$ and $b$, as well as that it depends only on absolute values $|a^*|=(aa^*)^{1/2}$ and $|b^*|=(bb^*)^{1/2}$.

\begin{lemma}\label{Ambiental}
    Let $A\subseteq B$ be two $C^*$-algebras and let $a$, $b\in A$. Then $a\perp^Sb$ in $A$ iff $a\perp^S b$ in $B$.
\end{lemma}

\begin{remark}
 In other words, strong $BJ$-orthogonality is an intrinsic property of the elements $a$, $b$ and it does not depend on ambient algebra.
\end{remark}

\begin{proof}
In \cite[Theorem 2.5.]{ArambasicAFA2014} it was proved that in any Hilbert $C^*$-module $a\perp^Sb$ if and only if $a\perp b\skp ba$ -- the basic, i.e.\ non strong $BJ$-orthogonality. In the case when the Hilbert module is algebra itself, the inner product is given by $\skp ba=b^*a$. Therefore
\begin{equation}\label{intrinsic}
a\perp^Sb\quad\mbox{iff}\quad a\perp bb^*a.
\end{equation}
The conclusion follows, because (non strong) $BJ$-orthogonality depends only on linear span of $a$ and $bb^*a$.
\end{proof}

\begin{lemma}\label{modul*}
We have
$a\strongp b$ if and only if $|a^*|\strongp|b^*|$.
\end{lemma}

\begin{proof}
Indeed, since $|b^*||b^*|^*=(bb^*)^{1/2}(bb^*)^{1/2}=bb^*$ we immediately have from (\ref{intrinsic}) that $a\perp^S b$ iff $a\perp^S|b^*|$.

For the first argument, let us consider $a$, $b\in A$ and the enveloping $W^*$-algebra of $A$ denoted by $A''$. There is a partial isometry $v\in A''$ such that $a=|a^*|v$ and also $|a^*|=av^*$. Suppose $|a^*|\strongp b$. Then for all $c\in A\subseteq A''$ we have
$$\|a\|=\||a^*|v\|\le\||a^*|\|\le \||a^*|+bcv^*\|=\|(a+bc)v^*\|\le\|a+bc\|,$$
from which we conclude $a\strongp b$. Similarly, if $a\strongp b$ and $c\in A$ arbitrary, then
$$\||a^*|\|=\|av^*\|\le\|a\|\le \|a+bcv\|=\|(|a^*|+bc)v\|\le\||a^*|+bc\|,$$
from which we conclude $a\strongp|b^*|$. Thus, $a\strongp b$ iff $|a^*|\strongp b$.
\end{proof}

\section{Isolated vertices}\label{Isolated}

In this section we characterize isolated vertices of the orthograph $\Gamma(A)$ in both unital and nonunital $C^*$-algebras $A$.

We need the notion of approximate invertibility, introduced in \cite{esmeral2021approximately} in the following way

\begin{definition} Let $A$ be a $C^*$-algebra and let $a\in A$. We say that $a$ is approximately right invertible if there is a net $\{b_i\}_{i\in I}$ such that $bb_i$ is approximate unit for $A$.
\end{definition}

Some results given in \cite{esmeral2021approximately} are stated in the following two propositions.

\begin{proposition}\label{Esm1} Let $A$ be a unital $C^*$-algebra, and let $a\in A$. Then $a$ is right invertible if and only if $a$ is approximately right invertible.
\end{proposition}

\begin{proof}
    The proof is given in \cite[Proposition 2.7]{esmeral2021approximately} and the paragraph immediately after it.
\end{proof}

\begin{proposition}\label{Esm2} Let $A$ be (possibly nonunital) $C^*$-algebra, and let $a\in A$. The following are equivalent:

($i$) $a$ is approximately right invertible;

($ii$) $a$ does not belong to any right modular ideal;

($iii$) for any pure state $\rho$, we have $\rho(aa^*)>0$.
\end{proposition}

\begin{proof}
    This is essentially \cite[Theorem 3.6]{esmeral2021approximately}.

    In truth, this Theorem is stated for nonunital nonabelian $C^*$-algebras, and proved for nonunital $C^*$-algebras (nonabelian property was not used). However thoroughly reading the proof we can conclude that it remains valid in unital case, as well. For the convenience of the reader we give an outline for the unital case.

    ($ii$) and ($iii$) are equivalent in the most general case \cite[Theorem 5.3.5]{Murphy}.

    Suppose ($i$) does not hold. Then by Proposition \ref{Esm1} $a$ is not right invertible, and hence $a\in aA\subsetneqq A$. Therefore, $Aa^*=(aA)^*$ is a proper left ideal in $A$ and it is, by \cite[Theorem 5.3.3]{Murphy} contained in some maximal modular left ideal, say $J$. Then, $a\in aA\subseteq J^*$, and ($ii$) does not hold.

    Suppose ($ii$) does not hold, i.e.\ $a\in J$ for some maximal modular right ideal $J$. Then $aA\subseteq JA\subseteq J\subsetneqq A$. Hence $a$ is not invertible, i.e. ($i$) does not hold.
\end{proof}

Also, we need the following technical lemma.

\begin{lemma}\label{lema} Let $A$ be a $C^*$-algebra, and let $a$, $b\in A$. If there is a state $\rho$, such that $\rho(aa^*)=\|a\|^2$ and $\rho(bb^*)=0$, then $a\perp^S b$.
\end{lemma}

\begin{proof} For any $c\in A$ by Cauchy-Schwartz inequality we have $|\rho(bc)|^2\le\rho(bb^*)\rho(c^*c)=0$. Assume $\|a\|=1$ (divide $a$ by $\|a\|$ if necessary). We have
$$\|a+bc\|^2=\|(a+bc)(a+bc)^*\|\ge|\rho(aa^*+ac^*b^*+bca^*+bcc^*b^*)|.$$
However, $\rho(bca^*)=\rho(bcc^*b^*)=0$ as well as $\rho(ac^*b^*)=\overline{\rho(bca^*)}=0$, whereas $\rho(aa^*)=1$. Hence
$$\|a+bc\|^2\ge 1=\|a\|^2.$$
\end{proof}

The following proposition, our first result is the converse of R1.

\begin{proposition}
Let $A$ be a unital $C^*$-algebra, and let $0\neq a\in A$ be an element that is not right invertible. Then $a$ is not an isolated point of the orthograph.
\end{proposition}

\begin{proof} Let $a$ be not right invertible. Suppose $\|a\|=1$. Then, by Proposition \ref{Esm1}, $a$ is not approximately right invertible. Further, by Proposition \ref{Esm2}, there is a pure state $\rho$ such that $\rho(aa^*)=0$.

On the other hand, there is another pure state $\tau$ such that $\tau(aa^*)=1$. Consider $b=\sqrt{1-aa^*}$. We have $b\neq0$ for if $b-0$ then $aa^*=1$ and hence $a$ is right invettible -- a contradiction. Then
$$\tau(bb^*)=\tau(1-aa^*)=0,$$
as well as
$$\rho(bb^*)=\rho(1-aa^*)=1.$$

By Lemma \ref{lema}, we have $a\strongp b$ by considering state $\tau$ and also $b\strongp a$ by considering state $\rho$. Hence $b\strongperp a$.
\end{proof}

The following two propositions \ref{IsoNonUni1} and \ref{isolatedNONunital} give a complete answer in non unital $C^*$-algebras.

\begin{proposition}\label{IsoNonUni1} Let $A$ be $C^*$-algebra, and let $b\in A$ be approximately right invertible element. Then $b$ is an isolated point of the orthograph.
\end{proposition}

\begin{proof} Suppose $b\strongperp a$. Then $a\strongp b$ and hence for all $c\in A$ we have
$$\|a+bc\|\ge\|a\|.$$
Choose $c=-b_ia$, where $b_i$ is an approximate right inverse, i.e.\ $bb_i$ is an approximate identity. Then
$$\|a\|\le\|a-bb_ia\|\to0,$$
implying $a=0$.
\end{proof}

\begin{proposition}\label{isolatedNONunital} Suppose that $a\in A$ is not approximately right invertible. Then $a$ is not an isolated point of the orthograph.
\end{proposition}

\begin{proof}Assume $\|a\|=1$. If $a$ is not approximately right invertible, then by Proposition \ref{Esm2} there is a pure state $\rho$ such that $\rho(aa^*)=0$. Also, there is another pure state, say $\tau$, such that $\tau(aa^*)=1$. Since $\rho$ is pure, there is some $b'\ge0$, $\|b'\|=1$ on which $\rho$ attains its norm, i.e.\ $\rho(b')=1$. (The existence of such $b'$ is a direct consequence of Kadison's transitivity theorem \cite[3.13.2]{Pedersen}, see also the proof of \cite[3.13.14]{Pedersen}.)

Consider the element $b=((1-aa^*)b'(1-aa^*))^{1/2}$. Obviously, $b\ge0$. Although $A$ might not be unital, the last term is valid. Namely, $A$ is an ideal in its unitization $\tilde A$; hence $(1_{\tilde A}-x)y\in A$ for any $x$, $y\in A$.

Further,
$$\rho(b^2)=\rho((1-aa^*)b'(1-aa^*))=\rho(b'-aa^*b'-b'aa^*+aa^*b'aa^*).$$
Since, $\rho(aa^*)=0$, we have $\rho(aa^*b')=\rho(b'aa^*)=\rho(aa^*b'aa^*)=0$, which leads to $\rho(b^2)=\rho(b')=1$. As a consequence $\|b\|=1$, and in particular $b\neq0$. However $\tau(1-aa^*)=0$, which implies $\tau(b^2)=0$.

Therefore we can apply Lemma \ref{lema} to conclude $a\bot^S b$ as well as $b\bot^Sa$. Hence, $a$ is not an isolated point.
\end{proof}

\begin{example}
Let $A=C_0(0,1)$. Then $f\in A$ is an isolated point of the orthograph $\Gamma(C_0(0,1))$ if and only if $f(t)\neq0$ for all $t\in(0,1)$. Indeed, pure states on $A$ are of the form $f\mapsto f(t)$, $t\in(0,1)$. By Propositions \ref{IsoNonUni1} and \ref{isolatedNONunital} $f$ is an isolated point iff $f$ is approximately right invertible and by Proposition \ref{Esm2} this is equivalent to $\rho(f\bar f)>0$ for all pure states $\rho$, i.e.\ $f(t)\neq0$ for all $t\in(0,1)$.
\end{example}

\begin{example}
Let $A=K(H)$ -- the algebra of all compact operators on some Hilbert space $H$. Then $T\in A$ is an isolated point of the orthograph $\Gamma(K(H))$  if and only if $T$ has dense range. Indeed, all pure states on $K(H)$ are given by $T\mapsto\left<Tx,x\right>$. As in the previous example $T$ is not isolated point iff $\left<TT^*x,x\right>>0$ which is equivalent to $T^*$ is injective, i.e.\ $T$ has a dense range.
\end{example}

\section{Noncommutative topology}\label{NonComTop}

In this section we summarize results from four classic Akemann's papers \cite{Akemann1968}, \cite{Akemann1969}, \cite{Akemann1970} and \cite{Akemann1971} necessary for further presentation.

Let $A$ be a $C^*$-algebra, and let $\pi$ denote its universal representation on some Hilbert space $H$. It is well known that the bicommutant of $\pi(A)$, denoted by $\pi(A)''$ which we shall abbreviate to $A''$ can be identified as a Banach space to the second dual $A^{**}$ of $A$. This identification allows to equip $A^{**}$ with von Neumann algebra structure. This second dual, or bicommutant in the universal representation is often called \emph{the enveloping algebra}. In what follows, we shall denote the enveloping algebra by $A''$.

While $A$ need not to have any projection (for instance $A=C_0(0,1)$), its enveloping algebra, being a von Neumann algebra, possess plenty of projections. Among all projections, we distinguish open and closed.

\begin{definition}
    Let $p\in A''$ be a projection.

    a) we say that $p$ is open if it is a supremum of some increasing net $a_\alpha$ of positive elements from $A$;

    b) we say that $p$ is closed if $1-p$ is open.

    c) we say that $p$ is compact, if it is closed and there is some $a\in A$, $a\ge0$ such that $ap=p$. (Of course this makes sense only in nonunital $A$, since otherwise we can take $a=1_A$.)
\end{definition}

\begin{remark}
    Note the analogy with continuous functions. Namely, $E$ is an open set iff $\chi_E$ is a supremum of positive continuous functions, $E$ is compact iff it is dominated by some continuous function vanishing at infinity, etc.
\end{remark}

\begin{remark}
    Note also, that a spectral projection of some positive $a\in A$ is open (respectively closed) if it corresponds to open (r.\ closed) interval. In particular, projection on the kernel of $a$ (corresponding to $\{0\}$) is closed.
\end{remark}

Our main tool is the following noncommutative Urysohn's lemma.

\begin{theorem}\label{Urysohn}

Let $p$, $q\in A''$ be projections such that $p$ is compact, $q$ is closed and $pq=0$. Then there is an $a\in A$ such that $0\le a\le 1$, $ap=p$ and $aq=0$.
\end{theorem}

\begin{proof}
    This is \cite[Lemma III.1]{Akemann1971}.
\end{proof}

Beside this theorem we also need some technical results.

\begin{lemma}\label{TwoClosed}
    Let $p$ and $q\in A''$ be closed projections. If $\|p(q-p\wedge q)\|<1$ then $p\vee q$ is also closed.
\end{lemma}

\begin{proof}
    This is \cite[Theorem II.7]{Akemann1969}.
\end{proof}

Recall that a projection $p$ in some von Neumann algebra $A$ is called \emph{minimal}, if $pAp=\mathbb C p$, which, in this case, is equivalent to the fact that if $q\le p$ implies $q=0$ or $q=p$. The enveloping algebra $A''$ has a lot of minimal projections. Every minimal projection is closed, see \cite[Proposition II.4]{Akemann1969} and also \cite[Corollary 2]{Akemann1968}.

\begin{lemma}\label{PhiMinBij}
    There is a bijective correspondence between pure states on $A$ and minimal projections in $A''$ such that if minimal $p\in A''$ corresponds to pure state $\varphi$ then
    \begin{equation}\label{PhiMin}
    pap=\varphi(a)p
    \end{equation}
\end{lemma}

\begin{proof}
    Let $\varphi$ be a pure state on $A$. It can be uniquely extended to a normal state on $A''$. Since $A''$ is always unital, throughout the proof, $1$ will always stand for unit in $A''$.

    By \cite[Theorem 3.13.6]{Pedersen} there is a minimal projection $p$ such that $\varphi(1-p)=0$. The latter implies $\varphi(a(1-p))=\varphi((1-p)a)=0$ for all $a\in A$. By minimality of $p$ we have $pap=\lambda p$ for a suitable $\lambda\in\mathbb C$. However, $\varphi(a)=\varphi(pap)=\varphi(\lambda p)=\lambda$ and hence (\ref{PhiMin}).

    Let $p\in A''$ be minimal. Then $pap=\lambda p$ for some $\lambda\in\mathbb C$. Define $\varphi(a):=\lambda$, i.e.\ (\ref{PhiMin}) is definition of $\varphi$ for a given $p$. It is easy to verify, that $\varphi$ is linear, positive and of norm one, hence a state. Also, $\varphi(1-p)=0$. If there is another positive linear $\psi\le\varphi$, then also $\psi(1-p)=0$ and hence $\psi(a)=\psi(pap)=\psi(\varphi(a)p)=\varphi(a)\psi(p)$, which implies $\psi=c\varphi$. (Here, we identify $\varphi$, $\psi$ with their normal extensions to $A''$.)
\end{proof}

\begin{remark}
    The proof of the previous lemma can also be tracked through various statements from \cite{Pedersen}, e.g.\ 3.6.11, 3.10.7, 3.11.10 and 3.13.6. An alternative comprehensive proof is given for the convenience of the reader.
\end{remark}

\begin{lemma}
    Let $A$ be a $C^*$-algebra, let $A''$ be its enveloping algebra, and let $p\in A''$ be a minimal projection. Then $p$ is compact.
\end{lemma}

\begin{proof}
    Every minimal $p\in A''$ is closed as it was already mentioned. Also, by Lemma \ref{PhiMinBij} there is the unique pure state $\varphi_p$ on $A$ such that $\varphi_p(p)=1$ and $pap=\varphi_p(a)p$ for each $a\in A$. By Kadison transitivity theorem, there is a positive $a\in A$ such that $\|a\|=1$ and $\varphi_p(a)=1$. Hence, $pap=p$. From $C^{*}$-identity
    $$\|(1-a)^{\frac{1}{2}}p\|^2=\|p(1-a)p\|=\|p-pap\|=0,$$
    we obtain $(1-a)^{\frac{1}{2}}p=0$, so $(1-a)p=0$. So we proved $ap=p$ which by definition ensures that $p$ is compact.

    (As we have already mentioned, $1$ is the unit in $A''$.)
\end{proof}

\begin{lemma}\label{ProjLE}
    Let $a$ be a positive element of $C^*$-algebra $A$ $(a\neq0)$. Then, there is a minimal projection $p\in A''$ such that
    $$ap=p\qquad\mbox{and}\qquad p\le a/\|a\|.$$
\end{lemma}

\begin{proof}
    Suppose $\|a\|=1$, otherwise divide $a$ by $\|a\|$. There is a pure state $\varphi$ that attains its norm on $a$. Let $p\in A''$ be the minimal projection that corresponds to $\varphi$, i.e.\ $\varphi(p)=1$, $pbp=\varphi(b)p$ for all $b\in A$. Then, as in the previous proof, we can conclude $ap=p$, and also
    $$a-p=a-pap=p(1_{A''}-a)p\ge0.$$
\end{proof}

\begin{lemma}\label{VEE}
    Let $p$, $q\in A''$ be minimal projections. Then $p\vee q$ is closed. Furthermore, if $r$ is projection such that $pr=qr=0$, then $(p\vee q)r=0$ as well.
\end{lemma}

\begin{proof}
    If $p=q$ the statement is trivial. Otherwise, $p\wedge q=0$, for if not, then $p\wedge q$ would be a nontrivial projection less then a minimal one.

    Once again, we identify pure states on $A$ with their normal extensions to the enveloping algebra $A''$.

    Let $\varphi_p$ be (the unique) pure state associated to $p$, i.e.\ such that $\varphi_p(p)=1$ and $pap=\varphi_p(a)p$. We claim $\varphi_p(q)<1$. Suppose it is not, i.e.\ $\varphi_p(q)=1$. Then $pqp=p$, i.e.\ $0=p(1-q)p=((1-q)p)^*(1-q)p$ implying by $C^*$-identity $(1-q)p=0$. Therefore $qp=p$ which means $q\ge p$ which is, due to the minimality of $q$, impossible, unless $q=p$.

    Thus, $\varphi_p(q)<1$ and hence
    $$\|qp\|^2=\|pqp\|=\|\varphi_p(q)p\|<1.$$
    Therefore, $\|p(q-p\wedge q)\|=\|pq\|<1$, which implies $p\vee q$ is closed by Lemma \ref{TwoClosed}.

    For the second part, note that $pr=0$ is equivalent to $p\le 1-r$, as well as $qr=0$ is equivalent to $q\le 1-r$. It follows $\sup\{p,q\}=p\vee q\le 1-r$, i.e.\ $(p\vee q)r=0$.
\end{proof}

\section{The diameter of the orthograph}\label{Diameter}

In this section, we discuss the diameter of the orthograph $\Gamma(A)$.

Recall that the path is a finite sequence of vertices, each of its adjacent entries are connected, and that the length of a path is the number of edges. Also, the distance between two vertices from some connected component is the minimal length of a path that connects these two vertices. Finally, the diameter of the graph (and also of its connected component) is the maximal distance of arbitrary two vertices.

This question was discussed in \cite{ArambasicBJMA2020} for a variety of concrete $C^*$-algebras, namely for full matrix algebras $M_n(\mathbb C)$, $n\ge2$, the algebra $B(H)$ of all bounded operators on some infinite dimensional Hilbert space $H$ as well as the commutative algebra $C(K)$, with $K$ compact Hausdorff.

The answer given in \cite{ArambasicBJMA2020} is the following:

1) The algebra $M_2(\mathbb C)$ is somewhat specific. All nontrivial connected components consists of precisely two elements; more specific, of those linear operators of rank one with mutually orthogonal ranges. In particular, the diameter of all connected components is equal to $1$.

2) In all other discussed examples, there is only one non trivial connected component, with finite diameter; in fact less or equal to $4$. In the case of $M_3(\mathbb C)$ the diameter is exactly $4$, in all $M_n(\mathbb C)$, $n\ge 4$ and $B(H)$ the diameter is $3$, whereas for $C(K)$, it is $3$ if there is at least one point in $K$ with countable local basis, and $2$ if there is no points with countable local basis.

The result we will present in Theorem \ref{Ocena dijametra} excludes three small algebras, $\mathbb C\cong M_1(\mathbb C)$, $\mathbb C\oplus\mathbb C\cong M_1(\mathbb C)\oplus M_1(\mathbb C)$ and $M_2(\mathbb C)$, all of them unital. Due to their simplicity, we can describe their orthographs easily.

\begin{proposition}
The algebra $M_1$ has no nontrivial noninvertible elements and its projective space is a singleton. Therefore, its orthograph $\Gamma(M_1)$ has only one vertex and no edges.

The algebra $M_1\oplus M_1$ has, up to a scalar multiple, two nontrivial noninvertible elements, $(1,0)$ and $(0,1)$. They are mutually $BJ$-orthogonal. Therefore, its orthograph $\Gamma(M_1\oplus M_1)$, beside isolated points, has two vertices and one edge connecting them.

In the algebra $M_2$ all non invertible elements belong to some connected component in its orthograph with exactly two members.
\end{proposition}

\begin{proof}
The first and the second assertion are obvious, whereas the algebra $M_2$ has been examined in \cite[Proposition 2.4]{ArambasicBJMA2020}.
\end{proof}

\begin{lemma}\label{p ili q minus p}
    Let $p,q\in A''$ be two different minimal projections. Then $q'=p\vee q-p$ is also minimal.
\end{lemma}

\begin{proof}
    As in Lemma \ref{VEE}, $p\wedge q=0$. By the famous Kaplansky's formula, we have $q'=p\vee q-p\sim q-p\wedge q=q$. So, there is a partial isometry $v$ such that $q=v^*v$ and $q'=vv^*$. Further, since $v=vv^{*}v$, we have $q'=vv^{*}vv^{*}=vqv^{*}$, and hence
    $$q'A''q'=vqv^* A'' vqv^*\subseteq vq A'' qv^*=v(\mathbb Cq)v^*=\mathbb C q'.$$
    Thus, $q'$ is minimal by definition.
\end{proof}

\begin{proposition}\label{uslovB}
    Let $A$ be a $C^*$-algebra not isomorphic to $M_1$, $M_1\oplus M_1$ or $M_2$, and let $p$, $q\in A''$ be two minimal projections. Then, there is another minimal projection $r\in A''$ such that $pr=qr=0$.
\end{proposition}

\begin{proof}
    Denote $s=p\vee q$. By Lemma \ref{VEE}, $s$ is closed.

    First, suppose that $s\neq 1=1_{A''}$. Then $1-s$ is nontrivial open projection, and hence, there is a net of positive elements in $A$ with supremum $1-s$. In particular, there is a nontrivial positive $a\in A$ such that $a\le 1-s$. By Lemma \ref{ProjLE}, there is a minimal $r\le a\le 1-s$. Hence $sr=0$.

    It remains to prove that if $s=1$, then $A$ must be isomorphic to one of mentioned three small algebras. So, suppose $s=1$; then $sAs=A$. By Lemma \ref{p ili q minus p}, $s=p+q'$, with $pq'=0$, and both $p$ and $q'$ are minimal. For every $a\in A$ we have
    \begin{equation}\label{A decompose}
    a=pap+q'ap+paq'+q'aq'.
    \end{equation}
    We shall prove that all of these four summands are scalar multiple of fixed elements. Let $\varphi_p$ and $\varphi_{q'}$ be the corresponding pure states. Then $pap=\varphi_p(a)p$ and $q'aq'=\varphi_{q'}(a)q'$. It remains to examine $q'ap$ and $paq'$. In $A''$ we have polar decomposition $q'ap=v|q'ap|$ for suitable partial isometry $v$. However, $|q'ap|^2=paq'ap\in pA''p\cong\mathbb C p$, and hence $|q'ap|\in \mathbb C p$. Thus $|q'ap|=\lambda p$ for some $\lambda\in\mathbb C$, and also $q'ap=\lambda v p$. Denote $vp$ ba $w$. It is easy to verify $w^*w=p$ and $ww^*=q'$. We prove that $w$ is unique up to a unimodular multiple. Indeed, suppose $u$ is another element such that $u^*u=p$ and $uu^*=q'$, and consider $u^*w$. Since $wp=up=p$, we have $u^*w\in pA''p$ and hence $u^*w=\lambda p$ for some unimodular $\lambda\in\mathbb C$. Therefore, any two partial isometries that connect $p$ and $q$ are scalar multiple of each other. Hence, $q'ap=\lambda vp=\lambda w$ is a scalar multiple of a fixed element regardless of $a$.

    Thus, decomposition (\ref{A decompose}) leads to a monomorphism from $A$ to $M_2(\mathbb C)$. It is isomorphism, if all four entries are nontrivial for some $a\in A$. In this case, $A\cong M_2$. If, however, $q'ap=paq'=0$ for all $a\in A$, then $A$ is isomorphic to the set of all diagonal $2\times 2$ matrices, i.e.\ to $M_1\oplus M_1$. Finally, if there is no two distinct minimal projections in $A$, then $A\cong M_1$.
\end{proof}

\begin{lemma}\label{Orth}
    Let $a$, $b\in A$ be positive elements of norm one, and let there is a projection $p$ (not necessarily minimal) such that $pa=p$ and $pb=0$. Then $b\strongp a$. If, additionally, there is another projection $q$ such that $qa=0$, $qb=q$ then $a\strongperp b$.
\end{lemma}

\begin{remark}
The previous lemma is valid for both $p\in A$ and $p\in A''$ by Lemma \ref{Ambiental}.
\end{remark}

\begin{proof}
    For any $c\in A$ we have
    $$\|a+bc\|\ge\|p(a+bc)\|=\|pa\|=\|p\|=1=\|a\|.$$
    The second part follows directly from the first.
\end{proof}

\begin{theorem}\label{Ocena dijametra}
    Let $A$ be a $C^*$-algebra not isomorphic to $M_1$, $M_1\oplus M_1$ or $M_2$. Then all elements that are not approximately invertible make a single connected component of its orthograph. Moreover, the diameter of this component does not exceed $4$.
\end{theorem}

\begin{proof}
    Let $a$, $b\in A$ be elements that are not approximately invertible. Without loss of generality we assume $a,b\geq0$ (Lemma \ref{modul*}) and $\|a\|=\|b\|=1$. By Proposition \ref{Esm2}, there are pure states $\varphi_a$, $\varphi_b$ that annihilates $a^2$, $b^2$ respectively. Denote their support projections by $q_a$ and $q_b$.

    Also, let $p'_a$, $p'_b$ stands for spectral projections of $a$, $b$ corresponding to the closed set $\{1\}$. All $p'_a$, $p'_b$, $q_a$, $q_b$ are closed, and moreover $q$'s are minimal. Also $p'_aq_a=0$, $p'_bq_b=0$.

    We can reduce $p'_a$ and $p'_b$ to smaller, and moreover minimal projections $p_a$, $p_b$ with the same property: $p_aa=p_a$, $p_bb=p_b$ and $p_aq_a=p_bq_b=0$. Indeed, there is a pure state, say $\psi$ which attains its norm on $a$, i.e.\ $\psi(a)=1$. Let $p_a$ be the minimal projection that corresponds to $\psi$. We have $p_aap_a=p_a$ which implies $p_a\le a$ (see the proof of Lemma \ref{ProjLE}).

    By Proposition \ref{uslovB}, there is another minimal projection $r$ such that $q_ar=q_br=0$. It is obvious that in $A''$ we have
    $$p_a\strongperp q_a\strongperp r\strongperp q_b\strongperp p_b.$$
    However this relation is the resident of $A''$ and it should be smoothened to be back in $A$.

    We construct elements $c_1$, $c_2$ and $c_3\in A$ using a noncommutative version of Urysohn's lemma, Lemma \ref{Urysohn}. Since $q_a$ is compact by construction and $p_a\vee r$ is closed by Lemma \ref{VEE} there is $c_1\in A$ such that $c_1q_a=q_a$ and $c_1(p_a\vee r)=0$, i.e.\ $c_1p_a=c_1r=0$. In a similar way we can construct $c_2$ such that $c_2r=r$, $c_2q_a=c_2q_b=0$ and $c_3$ for which $c_3q_b=q_b$, $c_3r=c_3p_b=0$.

    By Lemma \ref{Orth} we have
    $$a\strongperp c_1\strongperp c_2\strongperp c_3\strongperp b.$$
\end{proof}

\section{Diameter estimation of direct sum of $C^{*}$-algebras}\label{finite case}

Let $A$ be a $C^*$-algebra. By $PS(A)$ we denote the space of all pure states on $A$. For two states $\rho:A\to\mathbb{C}$ and $\tau:B\to\mathbb{C}$ on $C^{*}$-algebras $A$ and $B$ we define a mapping $\rho\oplus\tau:A\oplus B\to\mathbb{C}$ by $$(\rho\oplus\tau)(a\oplus b)=\rho(a)+\tau(b).$$
It is easy to prove that if $\rho$ and $\tau$ are states (on $A$ and $B$, respectively), the same is true for $\rho\oplus0$ and $0\oplus\tau$ (on $A\oplus B$).

\begin{lemma}
    Let $A$ and $B$ be $C^*$-algebras, and let $C=A\oplus B$. Then $$PS(C)=(PS(A)\oplus 0)\cup(0\oplus PS(B)).$$
\end{lemma}

\begin{proof}
    First, we prove that if $\rho$ is a pure state on $A$ then $\rho\oplus0$ is a pure state on $C$. Assume  $\rho\oplus0=\lambda\varphi+(1-\lambda)\psi$ for some $\lambda\in(0,1)$ and some states $\varphi$, $\psi$ on $C$. Define $\varphi_1:A\to\mathbb{C}$ and $\psi_1:A\to\mathbb{C}$ by $\varphi_1(a)=\varphi(a\oplus0)$ and $\psi_1(a)=\psi(a\oplus0)$. Similarly, define $\varphi_2:B\to\mathbb{C}$ and $\psi_2:B\to\mathbb{C}$ by $\varphi_2(b)=\varphi(0\oplus b)$ and $\psi_2(b)=\psi(0\oplus b)$. All these mappings are positive functionals, and it holds
    $$\rho=\lambda\varphi_1+(1-\lambda)\psi_1\quad \text{and}\quad 0=\lambda\varphi_2+(1-\lambda)\psi_2.$$
    Obviously, $\varphi_2=\psi_2=0$ and hence $\|\varphi_1\|=\|\psi_1\|=1$. Therefore $\varphi_1$ and $\psi_1$ are states. Since $\rho$ is pure, we have $\varphi_1=\psi_1=\rho$.

    Similarly, we conclude that $0\oplus\tau$ is a pure state on $C$ if $\tau$ is pure on $B$.

    Now suppose that $\rho$ is a pure state on $C$. Define $\rho_1:A\to\mathbb{C}$ and $\rho_2:B\to\mathbb{C}$ by $\rho_1(a)=\rho(a\oplus0)$ and $\rho_2(b)=\rho(0\oplus b)$. Suppose
    \begin{equation}\label{5.1suprotno}
    \|\rho_1\oplus0\|\neq0,\qquad\|0\oplus\rho_2\|\neq0.
    \end{equation}
    Then we can write
    %
    %
    $$\rho=\|\rho_1\oplus0\|\left(\frac{1}{\|\rho_1\oplus0\|}(\rho_1\oplus0)\right)+\|0\oplus\rho_2\|\left(\frac{1}{\|0\oplus\rho_2\|}(0\oplus\rho_2)\right).$$
    Note that the functionals in brackets are of norm one -- hence states. Also, from $1=\|\rho\|$ and $\rho(a\oplus b)=\rho(a\oplus0)+\rho(0\oplus b)=\rho_1(a)+\rho_2(b)$, we see that $\|\rho_1\oplus0\|+\|0\oplus\rho_2\|=1$, so by definition of pure state, exactly one of $\|\rho_1\oplus0\|$ and $\|0\oplus\rho_2\|$ is equal to zero, which violates (\ref{5.1suprotno}). Then either $\|\rho_1\oplus0\|=0$ and hence $\rho=0\oplus\rho_2$, or $\|0\oplus\rho_2\|=0$ and consequently $\rho=\rho_1\oplus0$.
\end{proof}

\begin{lemma}
    If $a_1\strongperp a_2$ and $b_1\strongperp b_2$ then $(a_1,b_1)\strongperp(a_2,b_2)$. In particular
    \begin{equation}\label{a00b}
        (a,0)\strongperp(0,b)
    \end{equation}
    whenever $a$, $b\neq0$.
\end{lemma}

\begin{proof}
Let $c,d\in A$. Then
\begin{equation}\nonumber
    \begin{split}
    \|(a_1,b_1)+(a_2,b_2)(c,d)\|&=\|(a_1+a_2c,b_1+b_2d)\|\\
    &=\max\{\|a_1+a_2c\|,\|b_1+b_2d\|\}\\
    &\geq\max\{\|a_1\|,\|b_1\|\}=\|(a_1,b_1)\|,
    \end{split}
\end{equation}
because $\|a_1+a_2c\|\geq\|a_1\|$ and $\|b_1+b_2d\|\geq\|b_1\|$ for every $c,d\in A$. So, we proved that $(a_1,b_1)\strongp(a_2,b_2)$ and the other direction is proved analogously.

To prove (\ref{a00b}) note that $a\strongperp0$ for any $a$. (Usually, $0$ is excluded from the orthograph, since otherwise everything will be connected. However, pairs $(a,0)$, $(0,b)\neq0$ for $a$, $b\neq0$.)
\end{proof}

\begin{remark}\label{ImpRed}
   Note that if $(a_1,b_1)\strongperp(a_2,b_2)$ it need not be neither $a_1\strongperp a_2$ nor $b_1\strongperp b_2$. Indeed,  let $I$ and $P$ be, respectively, unit matrix and projection of rank one in $M_2(\mathbb{C})$. Then $(I,P)\strongperp(P,I)$, but $I\strongperp P$ is not true, since $I\strongp P$ and $P\not\strongp I$.
\end{remark}

\begin{lemma} Let $A$ and $B$ be (nontrivial) $C^*$-algebras, and let $C=A\oplus B$. Then
\begin{enumerate}
\item\label{Lema2} $A\oplus B$ is unital if and only if both $A$ and $B$ are unital.

\item\label{Lema3} Let $A$ and $B$ are unital. $(a,b)$ is not right invertible in $A\oplus B$ if and only if at least one of $a$, $b$ is not right invertible in $A$, $B$ respectively.

\item\label{Lema4} $(a,b)$ is not approximately right invertible in $A\oplus B$ if and only if at least one of $a$, $b$ is not approximately right invertible in $A$, $B$ respectively.
\end{enumerate}
\end{lemma}

\begin{proof} (\ref{Lema2}) and (\ref{Lema3}) are trivial, so we give a proof only for (\ref{Lema4}).

If $(a,b)$ is not approximately right invertible, then by Proposition \ref{Esm2} there exist pure state $\rho$ such that $\rho((a,b)(a,b)^*)=0$. By characterization of pure states in previous lemma, there exists a pure state $\rho_1$ on $A$, or a pure state $\rho_2$ on $B$ such that $(\rho_1\oplus0)((a,b)(a,b)^*)=0$ or $(0\oplus\rho_2)((a,b)(a,b)^*)=0$. So, $\rho_1(aa^*)=0$ or $\rho_2(bb^*)=0$ and then at least one of $a$ and $b$ is not approximately right invertible by Proposition \ref{Esm2}.

For the other direction, suppose $a$ is not approximately right invertible. Then there exists a pure state $\rho_1$ on $A$ such that $\rho_1(aa^*)=0$. But than $(\rho_1\oplus0)((a,b)(a,b)^*)=(\rho_1\oplus0)((aa^*,bb^*))=0$. However $\rho_1\oplus0$ is pure; hence $(a,b)$ is not approximately right invertible.
\end{proof}

\begin{proposition} Let $A$ and $B$ be (nontrivial) $C^*$-algebras not isomorphic to $M_2(\mathbb C)$, and let $C=A\oplus B$. Then, there is only one nontrivial connected component in the orthograph of $C$, and for its diameter $\diam(C)$ there holds
$$\diam(C)\le\max\{\diam(A),\diam(B),3\}.$$
\end{proposition}

\begin{proof} Let $(a_1,b_1)$, $(a_2,b_2)\in C$ be non isolated points, hence non approximately right invertible. Then at least one of $a_1$, $b_1$ is not approximately right invertible, as well at least one of $a_2$, $b_2$. Let us consider the following possibilities:

1. Let $a_1$, $b_2$ not be approximately right invertible. Suppose $a_1\neq0$, $b_2\neq0$. Then $a_1$, $b_2$ are not isolated points in $A$, $B$ respectively. Therefore, there are $a'\in A$ and $b'\in B$ such that $a_1\strongperp a'$ and $b'\strongperp b_2$. Then we have the following path of length $3$:
\begin{equation}\label{PathAB}
(a_1,b_1)\strongperp(a',0)\strongperp(0,b')\strongperp(a_2,b_2).
\end{equation}

Note that nothing is lost if one or both of $a_1$, $b_2$ is equal to $0$, i.e.\ the same path (\ref{PathAB}) holds (for arbitrary $a'$ and/or $b'$), and even more, it can be shortened to $(0,b_1)\strongperp(a_2,0)$ if $a_1=0$, $b_2=0$.

2. Let $b_1$ and $b_2$ not be approximately right invertible, and suppose $0\neq b_1\neq b_2\neq0$. Since $B$ is not isomorphic to $M_2(\mathbb C)$, it has only one non trivial connected component. Thus we have a path of length $k\le\diam(B)$ in $B$:
$$b_1\strongperp b^{(1)}\strongperp\dots\strongperp b^{(k-1)}\strongperp b_2.$$
Then there is a path in $A\oplus B$
$$(a_1,b_1)\strongperp(0,b^{(1)})\strongperp\dots\strongperp(0,b^{(k-1)})\strongperp(a_2,b_2).$$
that connects $(a_1,b_1)$ and $(a_2,b_2)$, also of length $k$.

If $b_1=b_2$ or $b_1=0$ or $b_2=0$, or all of them, then $b_1$ and $b_2$ can be connected by a path of length $2$ or $1$. If $b_1=b_2$, we have $b_1\strongperp b'\strongperp b_2=b_1$ for some $b'\in B$, $b'\strongperp b_1$, whether $b_1=0$ or not. If $b_1=0$ then $b_1=0\strongperp b_2$ and similarly for $b_2=0$.

3. The possibility $a_1$, $a_2$ are not approximately invertible in $A$ is considered similarly.
\end{proof}

\begin{remark}
Note that the previous proof works even if one summand is $M_1\equiv\mathbb C$. Indeed, if $C=A\oplus M_1$, then $b$'s can be either $0$ or invertible, and the proof is perfectly fine.
\end{remark}

\begin{corollary} Let $A=M_{n_1}(\mathbb C)\oplus M_{n_2}(\mathbb C)\oplus\dots\oplus M_{n_m}(\mathbb C)$ be finite dimensional $C^{*}$-algebra, such that $n_j\neq2,3$ for all $1\le j\le m$. Then the diameter of its orthograph is at most $3$.
\end{corollary}

\begin{proof}Indeed, by \cite[Proposition 2.7]{ArambasicBJMA2020} the diameter of all summands $M_{n_j}(\mathbb C)$ does not exceed $3$. The result follows by previous Propositon.
\end{proof}

\section{Questions}\label{questions}

At the end of this note we raise two questions for future investigation.

1. Determine the exact diameter of the orthograph of an arbitrary finite dimensional $C^*$-algebra. A specific difficulty of this question is the impossibility to reduce it from $A\oplus B$ to its summands -- see Remark \ref{ImpRed}.

2. It would be interesting to reformulate the conclusion of Proposition \ref{uslovB} in some other terms. More precisely, do the following conditions are mutually equivalent?
    \begin{enumerate}
        \item\label{B0} For any $p$, $q\in A''$ minimal, there is minimal $r$ such that, $pr=qr=0$, i.e.\ the conclusion of Proposition \ref{uslovB};
        \item\label{B1} Any two maximal modular left ideals in $A$ have non empty intersection;
        \item\label{B2} The orthogonal dimension of $A$ is at least $3$. (Here, the orthogonal dimension is the maximal cardinality of some family of projections $\{p_i\}_{i\in I}$ such that $p_ip_j=0$ for $i\neq j$);
        \item\label{B3} For any two pure states on $A$ there is another pure state $BJ$-orthogonal to them.
    \end{enumerate}

\begin{remark}
    It would be interesting also, to reformulated the previous conditions in more terms, more specific, in terms of irreducible representations.
\end{remark}

A partial answer can be given as follows:

\begin{proposition}
    We have (\ref{B1}) implies (\ref{B3}) and (\ref{B1}) is equivalent to (\ref{B0}).
\end{proposition}

\begin{proof}
    (\ref{B1}) $\Rightarrow$ (\ref{B3}).

    Let $\varphi$ and $\psi$ be pure states on $A$, and let $N_\varphi$ and $N_\psi$ be their left ideals, defined by
    $$N_\varphi=\{a\in A\mid\varphi(ab)=0,\mbox{ for all }b\in A\},$$
    and similarly for $N_\psi$. By (\ref{B1}), the intersection $C=N_\varphi\cap N_\psi$ is not empty. $C$ is ideal, hence algebra. Pick some $c>0$, $c\in C$. There is a pure state $\tau$ on $C$, such that $\tau(c)=1$.

    It is enough to prove that $\tau$ can be extended to whole $A$. (We can assume that $A$ has a unit). There is a positive extension of $\tau$ of norm $1$ by Hahn-Banach theorem. Such an extension is a state, and we can prove that it is a pure state as well, by standard procedure. Namely, the set $E$ of all these extensions is closed (hence compact) and convex in the set of all states on $A$. By Krein-Millman theorem there is an extreme points in $E$, say $\tau'$. Suppose $\tau'$ is not pure on $A$. That means that $\tau'$ is not an extreme point in the set of all states (not only extending $\tau$) on $A$. Therefore $\tau'=\theta\sigma_1+(1-\theta)\sigma_2$ for some other states $\sigma_1$, $\sigma_2$, and also
    $$\tau=\tau'|_C=\theta\sigma_1|_C+(1-\theta)\sigma_2|_C.$$

    The state $\tau$ is pure on $C$, so $\sigma_1|C=\sigma_2|_C=\tau$, and hence $\sigma_1$, $\sigma_2\in E$. Finally, $\tau'$ is extreme point in $E$, and hence $\sigma_1=\sigma_2=\tau$.

    (\ref{B1}) $\Leftrightarrow$ (\ref{B0}).

    From Lemma \ref{PhiMinBij}, we know that for $p,q\in A''$ minimal projections, there exist pure states $\varphi_p$ and $\varphi_q$ such that $\varphi_p(p)=\varphi_q(q)=1$ and $pap=\varphi_p(a)p$ and $qaq=\varphi_q(a)q$ for every $a\in A$. Note that $pr=qr=0$ is equivalent to $prp=qrq=0$, so that it must be $\varphi_p(r)=\varphi_q(r)=0$. Thus, $pr=qr=0$ which is equivalent to $r\in Ker(\varphi_p)\cap Ker(\varphi_q)$.
\end{proof}

\bibliographystyle{abbrv}
\bibliography{Orthograph2023Feb}

\end{document}